\pdfoutput=1
\RequirePackage{ifpdf}
\ifpdf 
\documentclass[pdftex]{sigma}
\else
\documentclass{sigma}
\fi

\numberwithin{equation}{section}

\newtheorem{Theorem}{Theorem}[section]
\newtheorem{Corollary}[Theorem]{Corollary}
\newtheorem{Lemma}[Theorem]{Lemma}
\newtheorem{Proposition}[Theorem]{Proposition}
 { \theoremstyle{definition}
\newtheorem{Example}[Theorem]{Example} }

\usepackage{eucal}

\begin{document}

\allowdisplaybreaks

\newcommand{\arXivNumber}{1512.06765}

\renewcommand{\PaperNumber}{060}

\FirstPageHeading

\ShortArticleName{Modular Form Representation for Periods of Hyperelliptic Integrals}

\ArticleName{Modular Form Representation\\ for Periods of Hyperelliptic Integrals}

\Author{Keno EILERS}

\AuthorNameForHeading{K.~Eilers}

\Address{Faculty of Mathematics, University of Oldenburg,\\
 Carl-von-Ossietzky-Str.~9-11, 26129 Oldenburg, Germany}
\Email{\href{mailto:keno.eilers@uni-oldenburg.de}{keno.eilers@uni-oldenburg.de}}

\ArticleDates{Received December 22, 2015, in f\/inal form June 17, 2016; Published online June 24, 2016}

\Abstract{To every hyperelliptic curve one can assign the periods of the integrals over the holomorphic and the
meromorphic dif\/ferentials. By comparing two representations of the so-called \emph{projective connection}
it is possible to reexpress the latter periods by the f\/irst. This leads to expressions including only the
curve's parameters~$\lambda_j$ and modular forms. By a change of basis of the meromorphic dif\/ferentials one
can further simplify this expression. We discuss the advantages of these explicitly given bases,
which we call \emph{Baker} and \emph{Klein basis}, respectively.}


\Classification{14H42; 30F30}

\section{Introduction}
Expressions of periods of hyperelliptic integrals in terms of $\theta$-constants have a long history, star\-ting with Rosenhain (1851) and Thomae (1866). Such a representation of second-kind dif\/ferentials was a part of the program by Felix Klein (1886, 1888) of constructing Abelian functions in terms of multi-variable $\sigma$-functions. Leaving aside considerations of expressions for higher genera analogues of the elliptic periods $2\omega$ in terms of $\theta$-constants, in this note we focus on expressions for the higher genera periods $2\eta$, which in the Weierstrass theory is of the form
\begin{gather*}
\eta=- \frac{1}{12\omega} \left(\frac{\vartheta_2''(0)}{\vartheta_2(0)}+ \frac{\vartheta_3''(0)}{\vartheta_3(0)}+ \frac{\vartheta_4''(0)}{\vartheta_4(0)}\right).
\end{gather*}
We present below an analogue of such an expression for hyperelliptic curves.

This problem is interesting both for theoretical reasons and computational viewpoints for developing black box calculation methods for these quantities. In recent times, the $\sigma$-function has come into focus in many applications, e.g., in theoretical physics and mathematics (see, e.g.,~\cite{ehkkl11} and references therein).

\section{The method}\label{section2}
Consider a genus-$g$ hyperelliptic curve
$\mathcal{C}\colon f(x) -y^2=0$, with $f(x)=\sum\limits_{i=0}^{2g+2}\lambda_i\cdot x^i$, $\lambda_i\in\mathbb{C}$,
and a canonical homology basis $\left(\mathfrak{a}_1,\ldots,\mathfrak{a}_g,\mathfrak{b}_1,\ldots,\mathfrak{b}_g\right)$, $\mathfrak{a}_i\circ \mathfrak{a}_j=\varnothing$, $\mathfrak{b}_i\circ \mathfrak{b}_j=\varnothing$, $\mathfrak{a}_i\circ \mathfrak{b}_j=\delta_{i,j}$ (see Fig.~\ref{hyperelliptic}).
\begin{figure}[h]
	\centering
		\includegraphics[width=0.5\textwidth]{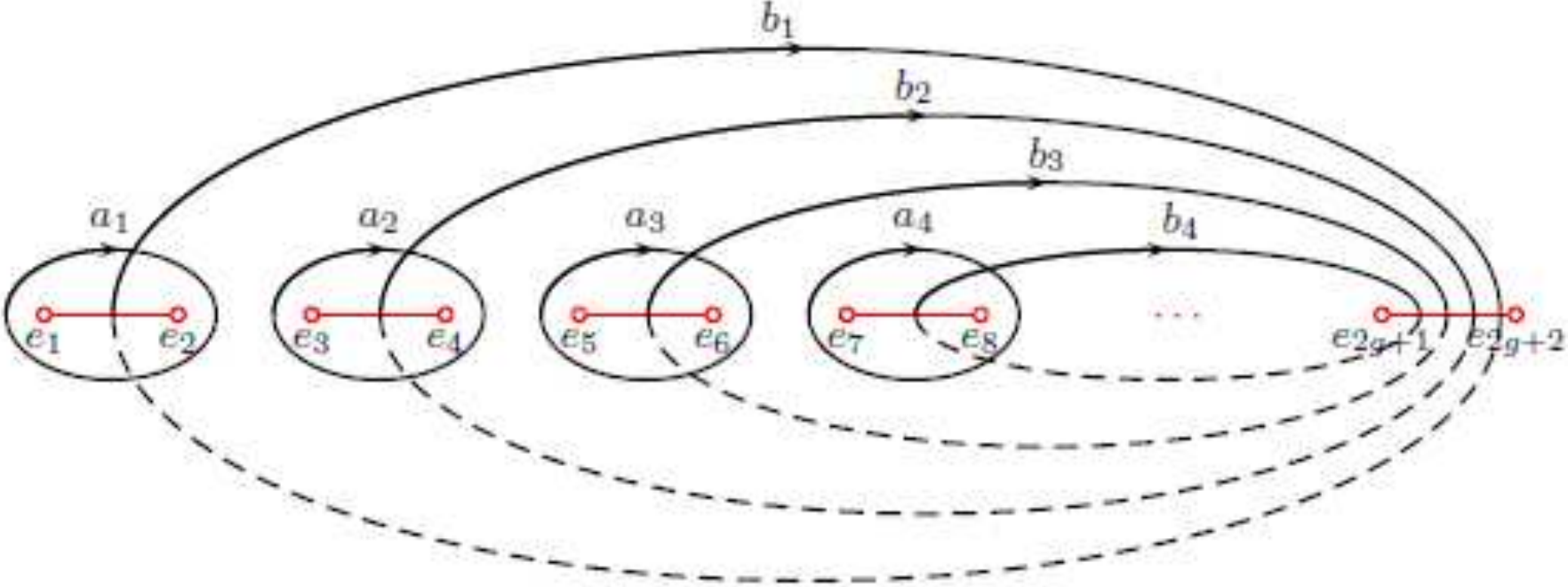}
	\caption{The canonical homology basis for a hyperelliptic curve.}
	\label{hyperelliptic}
\end{figure}
Then f\/ix a basis of $2g$ dif\/ferentials in the form
$\mathcal{B}=\frac1{y}\big(x^{i-1}\cdot {\mathrm d} x\big)_{1\leq i\leq2g}$, of which the f\/irst half are holomorphic and the second half are meromorphic. Though partly meromorphic, we refer to this basis as a cohomology basis. (The holomorphic part is in fact dual to the cycle basis.) This basis can be normalized in a form especially due to Baker~\cite{bak897}:
\begin{gather}
\mathcal{B'} =\frac1{y}\big(x^{i-1}\cdot {\mathrm d}x\big)_{1\leq i\leq g}\cup\frac1{4y}\left(\sum_{k=j}^{2g+1-j} (k+1-j )\lambda_{k+1+j} x^k\cdot {\mathrm d}x\right)_{1\leq j\leq g}\nonumber\\
\hphantom{\mathcal{B'}}{}
= (u_i )_{1\leq i\leq g}\cup (r_j )_{1\leq j\leq g},\label{Baker}
\end{gather}
of which again the f\/irst half are the (holomorphic) dif\/ferentials of the f\/irst kind and the second half are the (meromorphic) dif\/ferentials of the second kind. Def\/ining now $2\omega$, $2\omega'$ as the $\mathfrak{a}$- and $\mathfrak{b}$-periods of the integrals over the holomorphic dif\/ferentials, and $2\eta$, $2\eta'$ as the respective periods of the integrals over the meromorphic dif\/ferentials, these matrices fulf\/ill the generalized Legendre relation
\begin{gather*}
\eta^T\omega = \omega^T\eta,\qquad \eta^T\omega'-\omega^T \eta' =
\frac{\imath\pi}{2}, \qquad {\eta'}^T\omega'={\omega'}^T\eta',
\end{gather*}
which make this setup as a natural generalization of the theory of elliptic functions by Weierstrass. Analogously to the Weierstrass theory, we introduce
\begin{gather*}
\varkappa=\eta (2\omega )^{-1}, \qquad \varkappa^T=\varkappa.
\end{gather*}
Via $\varkappa$ we will calculate in the following the periods of second kind integrals
\begin{gather*}
\eta = 2\varkappa \omega, \qquad \eta' = 2\varkappa \omega' - \imath \pi {(2\omega)^T}^{-1}.
\end{gather*}
Using the Riemann period-matrix $\tau=\omega^{-1}\cdot\omega'$, we def\/ine the Riemann-$\theta$-function
\begin{gather*}
 \theta[\varepsilon](\boldsymbol{z};\tau)=\sum_{\boldsymbol{n}\in
 \mathbb{Z}^g}\exp \big\{ \imath\pi(\boldsymbol{n}+\boldsymbol{\varepsilon})^T\tau
 (\boldsymbol{n}+\boldsymbol{\varepsilon}) +2\imath\pi
 (\boldsymbol{n}+\boldsymbol{\varepsilon})^T (\boldsymbol{z}+\boldsymbol{\varepsilon'})\big\},
\end{gather*}
with half-integer characteristics \begin{gather*}
[\varepsilon]= \left[ \begin{matrix} \boldsymbol{\varepsilon}^T\\
 \boldsymbol{\varepsilon'}^T\end{matrix} \right] =
\left[\begin{matrix}
 \varepsilon_1&\ldots&\varepsilon_g \\
 \varepsilon_1'&\ldots&\varepsilon_g'
\end{matrix} \right].
\end{gather*}
We call a characteristic even (resp.~odd), if $4{\boldsymbol\varepsilon}^T{\boldsymbol\varepsilon}'\mod 2=0$ (resp.~1); the associated $\theta$-function inherits the parity of the characteristic. We also can def\/ine all $\theta$-constants: $\theta[\varepsilon]=\theta[\varepsilon](0)$, $\theta_{ij}[\varepsilon]=\frac{\partial^2}{\partial z_i\partial z_j}\theta[\varepsilon](z) \big|_{z=0}$ if $[\varepsilon]$ is even. Denote{\samepage
\begin{gather}N_g=\binom{2g+1}{g}, \label{ng} \end{gather} which is the number of non-singular even characteristics.}

We can clarify the role of the characteristics more explicitly by pointing out their connection to the Abelian images of the branching points $e_i$, which are in the hyperelliptic case the zeroes of~$f$. That is, there exists a characteristic $[\varepsilon_j]$, such that the following relation holds
\begin{gather*}
\boldsymbol{\mathfrak{A}}_j=\int_{R_0}^{(e_j,0)}{\boldsymbol u}=:2\omega{\boldsymbol \varepsilon_j}+2\omega'{\boldsymbol \varepsilon_j'},\qquad j=1,\ldots,2g+2.
\end{gather*}
We use the notation $[\varepsilon_j]=[\boldsymbol{\mathfrak{A}}_j]$, and hence the notion of even or odd branching points.
Now consider a partition of the branching points,
$\{i_1,\ldots,i_{g+1}\}\cup\{j_1,\ldots,j_{g+1}\}=\mathcal{I}_0\cup\mathcal{J}_0$ of $\{1,\ldots,$ $2g+2\}$.
To that partition corresponds a certain characteristic $[\varepsilon]$ by
\begin{gather}
\sum_{k=1}^{g+1}\boldsymbol{\mathfrak{A}}_{i_k}+\boldsymbol{K}_{R_0}=2\omega{\boldsymbol \varepsilon}+2\omega'{\boldsymbol \varepsilon'}, \label{characteristic}
\end{gather}
and that precise $[\varepsilon]$ is even and non-singular. $\boldsymbol{K}_{R_0}$ is the vector of Riemann constants for the base point $R_0$. The vector of Riemann constants of a hyperelliptic curve and with a branching point as the base point is the sum of all vectors $\boldsymbol{\mathfrak{A}}_j$, for which the respective characteristic $[\boldsymbol{\mathfrak{A}}_j]$ is odd (for further information, see \cite[p.~305]{fk980}).

As an example of these partitions, consider $g=2$ and f\/ix a base point $R_0=e_6$. Then $\left[\varepsilon_{45}\right]\equiv\left[\varepsilon_{456}\right]\equiv\left[\varepsilon_{123}\right]$ corresponds to the partition $\left\{1,2,3\right\}\cup\left\{4,5,6\right\}$. The last equivalence is due to Abel's theorem applied to the meromorphic function $y/(x-e_1)(x-e_2)(x-e_3)$: The divisors of zeros and poles, respectively, are linear equivalent and hence the Abel maps of these divisors coincide.

Now, we have the tools to introduce the bi-dif\/ferential $\Omega(P,Q)$ on $\mathcal{C}\times \mathcal{C}\ni (P,Q)$, which is called the {\em canonical bi-differential of the second kind} if it is
\begin{itemize}\itemsep=0pt
\item symmetric
$\Omega(P,Q)=\Omega(Q,P)$,
\item normalized at $\mathfrak{a}$-periods:
$ \oint_{\mathfrak{a}_k} \Omega(P,Q)=0$, $k=1,\ldots,g$,
\item and has the only pole of the second order along the
diagonal, namely it has the following expansion:
\begin{gather*}
 \Omega(P,Q) = \frac{{\mathrm d}\xi(P){\mathrm d}\xi(Q)} {(\xi(P)-\xi(Q))^2}+ \frac{1}{6}\mathfrak{S}(R) + \text{higher order terms},
\end{gather*}
where $\xi(P)$ and $\xi(Q)$ are local coordinates of points $P=(x,y)$ and $Q=(z,w)$ in the vicinity of a point $R, \xi(R)=0$ respectively.
\end{itemize}
The quantity $\mathfrak{S}(R)$ is called the {\it holomorphic projective connection} (see \cite[p.~19]{fay973}; note that this object is sometimes called Bergman projective connection, but here we adopt the notion of Fay). Our purpose is now to express $\mathfrak{S}(R)$ in two dif\/ferent ways, one containing $\varkappa$, equate them and solve for~$\varkappa$.

\section[Two representations of $\Omega(P,Q)$]{Two representations of $\boldsymbol{\Omega(P,Q)}$}\label{section3}
The canonical bi-dif\/ferential is uniquely def\/ined by the given conditions. But it has several representations, whose derivations are described in particular in~\cite{eee13}. We restrict this inspection to the {\it Fay-representation} and the {\it Klein--Weierstrass-representation}. Note that Fay is using a monic polynomial~$f$ (see \cite[p.~20]{fay973}), so in the following we set $\lambda_{2g+2}=1$. But this is no loss of generality, because we can rescale $f$ as necessary after calculating $\varkappa$, respectively all related objects. (In especially, for $f\rightarrow \lambda_{2g+2}\cdot f$ we have $\varkappa\rightarrow \lambda_{2g+2}\cdot\varkappa$, $\eta\rightarrow \sqrt{\lambda_{2g+2}}\cdot\eta$ and $\omega\rightarrow \omega/\sqrt{\lambda_{2g+2}}$.)

Introducing the normalized holomorphic dif\/ferentials ${\boldsymbol v}=(2\omega)^{-1}{\boldsymbol u}$ and non-singular odd characteristics $[\delta]$, we present two realizations
\begin{gather}
\Omega(P,Q) ={\mathrm d}_x {\mathrm d}_z \log\theta[\delta]\left(\int_Q^P{\boldsymbol v}\right)
=\frac{F(x,z)+2yw}{4(x-z)^2}\frac{{\mathrm d}x}{y} \frac{{\mathrm d}z}{w} + 2\sum_{i,j=1}^gu_i\varkappa_{ij}u_j.
\label{Bidifferential}
\end{gather}
The symmetric polynomial $F(x,z)$ is the Kleinian 2-polar,
\begin{gather*}F(x,z)=2\lambda_{2g+2}z^{g+1}x^{g+1}+\sum_{k=0}^gx^kz^k\big(2\lambda_{2k}+(x+z)\lambda_{2k+1}\big), \qquad F(x,x)=2f(x),\end{gather*}
which was introduced by Klein to represent the second kind bi-dif\/ferential in the above given form.

From the f\/irst and second expression of equation~\eqref{Bidifferential} we can derive two dif\/ferent representations, which we will cite in the following propositions:

\begin{Proposition}[\protect{\cite[p.~20]{fay973}}]
For a local coordinate $\xi(R)$ and for any non-singular even characteristic $[\varepsilon]$ corresponding to the partition $\{i_1,\ldots,i_{g+1}\}\cup\{j_1,\ldots,j_{g+1}\}=\mathcal{I}_0\cup\mathcal{J}_0$ of $\{1,\ldots,2g+2\}$, the projective connection is of the form
\begin{gather}
\mathfrak{S}_{\rm Fay}(R)= \{x,\xi(R) \}{\mathrm d}\xi^2+\frac{3}{8}\sum\limits_{\forall\, \mathcal{I}_0\cup \mathcal{J}_0} \left( \mathrm{d}\ln
\frac{\prod\limits_{ i\in \mathcal{I}_0} x-e_i}{\prod\limits_{j\in \mathcal{J}_0} x-e_j} \right)^2-6\sum_{i,j=1}^g\frac{\theta_{ij}[\varepsilon]}{\theta[\varepsilon]}v_i(R)v_j(R),
\label{fay}
\end{gather}
where $\left\{\cdot,\cdot\right\}$ is the Schwartzian derivative
\begin{gather*}\{x(\xi),\xi\} = {\frac{\mathrm{d}^3 x(\xi)
 /\mathrm{d}\xi^3}
{\mathrm{d} x(\xi) /\mathrm{d}\xi} } -
\frac32\left (\displaystyle{\frac{\mathrm{d}^2 x(\xi) /
 \mathrm{d}\xi^2} { \mathrm{d} x(\xi) /\mathrm{d}\xi}
 }\right)^2.
\end{gather*}
\end{Proposition}

The next result is taken from \cite[p.~307]{eee13}, where the projective connection was constructed for a larger class of curves:
\begin{Proposition}[\protect{\cite[p.~307]{eee13}}]
Write the curve $\mathcal{C}$ as $y^2=f(x)$, so that a prime $'$ means differentiation with respect to~$x$. Then
\begin{gather}
\mathfrak{S}_{\rm KW}(R)=\left\{x,\xi(R)\right\}{\mathrm d}\xi^2-\frac{3}{2y}y''{\mathrm d}x^2+6{\boldsymbol r}^T{\boldsymbol u} +12{\boldsymbol u}^T\varkappa{\boldsymbol u},\label{kw}
\end{gather}
all differentials evaluated at the point $R=(x,y)$.
\end{Proposition}

As stated above, $\mathfrak{S}_{Fay}$ and $\mathfrak{S}_{KW}$ coincide, so we can solve linearly for $\varkappa$. We see that the Schwartzian derivatives cancel and as can be seen in particular in the term ${\boldsymbol u}^T\varkappa{\boldsymbol u}$, only entries of $\varkappa$ along the anti-diagonals share the same order of $x$. Therefore we will solve order by order and for a f\/irst insight we will do so by expanding $x$ in terms of a local coordinate $\xi$.

\section{The results}\label{section4}
Consider a partition $\mathcal{I}_0\cup\mathcal{J}_0=\{1,\ldots,2g+2\}$, and denote with $S_j(\mathcal{I})$, $j=0,\ldots, g+1$ the elementary symmetric function of order $j$ built in the branching points $e_i$ with indices taken from $\mathcal{I}$, namely $S_1(\mathcal{I})= \sum\limits_{i\in \mathcal{I}} e_{i}$, $S_2(\mathcal{I})= \sum\limits_{i,k\in \mathcal{I};\, i\leq k} e_{i}e_k$, etc.

\begin{Example}
For any even $g=2$ hyperelliptic curve, $\varkappa$ is given as
\begin{gather}
\varkappa= \frac{1}{8}\left(
\begin{matrix}
S_3(\mathcal{I}_0)S_1(\mathcal{J}_0)+S_3(\mathcal{J}_0)S_1(\mathcal{I}_0)
& -S_3(\mathcal{I}_0)-S_3(\mathcal{J}_0)\\
-S_3(\mathcal{I}_0)-S_3(\mathcal{J}_0)
&S_2(\mathcal{I}_0)+	S_2(\mathcal{J}_0)
\end{matrix}\right)\nonumber\\
\hphantom{\varkappa=}{} -\frac{1}{2}(2\omega)^{-1^T}\frac{\left(
\begin{matrix}
\theta_{11}[\varepsilon]&\theta_{12}[\varepsilon]\\
\theta_{12}[\varepsilon]&\theta_{22}[\varepsilon]
\end{matrix}\right)
}{\theta[\varepsilon]}(2\omega)^{-1},\label{example2}
\end{gather}
with $\varepsilon$ according to the partition $\mathcal{J}_0= \{i,j,6 \}$. We observe that $\varkappa$ splits nicely into a transcendental part consisting of various $\theta$-constants, and a rational part, which will be inspected more closely below.
\end{Example}
After def\/ining the column-vectors $\left(\boldsymbol{V}_1,\ldots,\boldsymbol{V}_g\right)=(2\omega)^{-1}$ we use a shorter notation by setting
\begin{gather*}
 \partial_{\boldsymbol{V_k}} \theta[\varepsilon]\equiv \sum_{l=1}^g V_{k,l}\frac{\partial}{ \partial z_l}
\theta[\varepsilon](\boldsymbol{z})\vert_{\boldsymbol{z}=\boldsymbol{0}} =\Theta_k[\varepsilon],\\
 \partial_{\boldsymbol{V_k}\boldsymbol{V_j}} \theta[\varepsilon]\equiv \sum_{l,m=1}^g V_{k,l}V_{j,m}\frac{\partial^2}
{ \partial z_l \partial z_m} \theta[\varepsilon](\boldsymbol{z})\vert_{\boldsymbol{z}=\boldsymbol{0}}=\Theta_{k,j}[\varepsilon],\\
 \partial_{\boldsymbol{V_k}\boldsymbol{V_j}\boldsymbol{V_i}} \theta[\varepsilon]\equiv \sum_{l,m,n=1}^g V_{k,l}V_{j,m}V_{i,n}\frac{\partial^3}
{ \partial z_l \partial z_m \partial z_n} \theta[\varepsilon](\boldsymbol{z})\vert_{\boldsymbol{z}=\boldsymbol{0}}=\Theta_{k,j,i}[\varepsilon],\qquad
 \text{etc.},
\end{gather*}
and also $\theta[\varepsilon]=\Theta[\varepsilon]$.

To avoid establishing the correspondence between branching points and characteristics in the given homology basis, we sum over all $\binom{2g+1}{g}$ possible partitions. The calculation time may grow, but as a result $\varkappa$ is expressible in terms of the parameters $\lambda_j$ of the curve:
\begin{gather}
\varkappa=\frac{1}{8}\frac{1}{10}\left(
\begin{matrix}
4\lambda_2&\lambda_3\\\lambda_3&4\lambda_4	
\end{matrix}
\right)-\frac{1}{2}\frac{1}{10}\sum_{10 [\varepsilon]}\frac{\left(
\begin{matrix}
\Theta_{11}[\varepsilon]&\Theta_{12}[\varepsilon]\\
\Theta_{12}[\varepsilon]&\Theta_{22}[\varepsilon]
\end{matrix}\right)}{\Theta[\varepsilon]}.\label{kappa2}
\end{gather}
The same technique is applicable for $g=3$, but with the dif\/ference, that the elements of the anti-diagonal share the same order in $\xi$. So at f\/irst, we only can derive the sum of the diagonal:
\begin{Example}
For any even $g=3$ hyperelliptic curve, $\varkappa$ is given as
\begin{gather}
\varkappa=\frac{1}{8}\left(\begin{matrix}
S_4(\mathcal{I}_0)S_2(\mathcal{J}_0){+}S_4(\mathcal{J}_0)S_2(\mathcal{I}_0)& -S_4(\mathcal{I}_0)S_1(\mathcal{J}_0){}-S_4(\mathcal{J}_0)S_1(\mathcal{I}_0) & X\\
-S_4(\mathcal{I}_0)S_1(\mathcal{J}_0){-}S_4(\mathcal{J}_0)S_1(\mathcal{I}_0) & Y & -S_3(\mathcal{I}_0){-}S_3(\mathcal{J}_0)\\
X & -S_3(\mathcal{I}_0){-}S_3(\mathcal{J}_0) & S_2(\mathcal{I}_0){+}S_2(\mathcal{J}_0)
\end{matrix}\right) \nonumber\\
\hphantom{\varkappa=}{}
-\frac1{2}\frac{\left(
\begin{matrix}
\Theta_{11}[\varepsilon]&\Theta_{12}[\varepsilon]&\Theta_{13}[\varepsilon]\\
\Theta_{12}[\varepsilon]&\Theta_{22}[\varepsilon]&\Theta_{22}[\varepsilon]\\
\Theta_{13}[\varepsilon]&\Theta_{23}[\varepsilon]&\Theta_{33}[\varepsilon]
\end{matrix}\right)}{\Theta[\varepsilon]},\label{example3}
\end{gather}
with $2X+Y=S_3(\mathcal{I}_0)S_1(\mathcal{J}_0)+S_3(\mathcal{J}_0)S_1(\mathcal{I}_0)$.
$[\varepsilon]$ here corresponds to the partition $\mathcal{J}_0= \{i,j,k,8 \}$.
\end{Example}

Again, we can sum over all allowed $[\varepsilon]$. If we do so, we f\/ind that the symmetric functions $2X+Y$ will sum to $20\lambda_4$. Examining dif\/ferent (even, $g=3$ and hyperelliptic) curves numerically, it is reasonable to assume that each entry depends linearly on a single $\lambda_i$, i.e., there are no additional constants present. (Please note that this is at this point only an assumption, but the sum of $20\lambda_4$ forbids many other possibilities.) This means that the sum over the entries~$X$ or~$Y$ are multiples of $\lambda_4$. If so, the respective prefactors of $\lambda_4$ for the~$X$- and~$Y$-sums are independent of the specif\/ic curve (of this type). Hence we f\/ind the right separation of $20\lambda_4$ by a numerical investigation of a concrete curve. That way we get up to numerical accuracy integer numbers
\begin{align}
\varkappa=\frac1{8}\frac1{35}\left(
\begin{array}{ccc}
15\lambda_2&5\lambda_3&\lambda_4\\
5\lambda_3&18\lambda_4&5\lambda_5\\
\lambda_4&5\lambda_5&15\lambda_6
\end{array}
\right)-\frac1{2}\frac1{35}\sum_{35 [\varepsilon]}\frac{\left(\Theta_{ab}[\varepsilon]\right)_{ a,b=1,\ldots,3}}{\Theta[\varepsilon]}.
\label{kappa3}
\end{align}
In especially the two ``partition numbers'' at $X$ and $Y$ are $1$ and $18$. This summed version of~$\varkappa$ relies on two numerically justif\/ied assumptions (linearity in~$\lambda_i$ and integer solutions). The examination of higher genera will reveal a number pattern, which adds further indication to that structure. But a full prove will be given below in Section~\ref{section6}, where we develop a method to disentangle these partition numbers. For that we f\/irst need some statements about the structure of~$\varkappa$ for arbitrary genus.

\section{The general pattern}\label{section5}
We want to generalize the representation of $\varkappa$ as it is in equations~\eqref{kappa2} and~\eqref{kappa3} to arbitrary genus. From the above examples follows the general pattern for the $\varkappa$ matrix: Let $\mathcal{I}_0=\{i_1,\ldots,i_{g+1}\}$ be a selection of $2g+2$ elements $1,\ldots,2g+2$ and $[\varepsilon]$ the associated non-singular even characteristic~\eqref{characteristic}. Then the following Ansatz is suggested for arbitrary genus~$g$
\begin{gather}
\varkappa=\frac1{8}\Lambda_g^{[\varepsilon]}
-\frac12\frac{(\Theta_{ab}[\varepsilon])_{a,b=1,\ldots,g}}{\Theta[\varepsilon]}. \label{kappaAnsatz}
\end{gather}
Summed over all $N_g$ non-singular even characteristics (given in \eqref{ng}) and divided by their number we get symmetric expressions independent of the special characteristic and f\/ixed homology basis. In that sense, the following theorem mainly states the splitting of $\varkappa$ into a modular part and a ``residual'' matrix~$\Lambda_g$, which we will get rid of in the last section.

\begin{Theorem}\label{maintheorem}
Let $\mathcal{C}$ be a hyperelliptic curve of genus $g$ and introduce Baker's $2g$-dimensional basis for the singular cohomology $($equation~\eqref{Baker}$)$. Let $2\omega$, $2\omega'$, $2\eta$, $2\eta'$ be period $g\times g$ matrices satisfying the generalized Legendre relation. Then $\varkappa=\eta(2\omega)^{-1}$ is of the form
\begin{gather}
\varkappa=\frac1{8}\frac1{N_g}\Lambda_g
-\frac1{2}\frac1{N_g}\sum_{N_g [\varepsilon]}\frac{(\Theta_{ab}[\varepsilon])_{a,b=1,\ldots,g}}{\Theta[\varepsilon]}, \label{kappaAnsatz2}
\end{gather}
with
\begin{gather*} \Lambda_g= \sum_{N_g [\varepsilon]} \Lambda_g^{[\varepsilon]}. \end{gather*}
\end{Theorem}
We conjecture that the matrix $\Lambda_g$ consists of the parameters $\lambda_j$ of the curve, together with integer coef\/f\/icients, the ``partition numbers'', as it was the case in genus~$2$ and (up to numerical uncertainty) in genus~$3$.

An anonymous referee pointed out a way to straighten the claims about $\Lambda_g$. We give his idea in form of a lemma:
\begin{Lemma}\label{reflemma}
The before defined matrix $\Lambda_g$ is expressible as
\begin{gather}
y^2{\boldsymbol u}^T\Lambda_g{\boldsymbol u}=\frac{N_g}{4g+2}\sum_{k=2}^{2g}\left(\frac{1}{2}k (2g+2-k )+\frac{1}{4} (2g+1 )\big( (-1 )^k-1\big)\right)\lambda_kx^{k-2}\mathrm{d}x^2.\label{refrelation}
\end{gather}
\end{Lemma}

Def\/ining the vector $\boldsymbol{X} = \big(1,x,\ldots, x^{g-1} \big)^T$ changes the left-hand side of equation~\eqref{refrelation} to $\boldsymbol{X}^T \Lambda_g\boldsymbol{ X} \mathrm{d}x^2$ and so it is clear, that the dif\/ferentials cancel.
\begin{proof}
We can combine the proof of Lemma~\ref{reflemma} and Theorem~\ref{maintheorem}.
From the comparison of equations~\eqref{fay} and~\eqref{kw} it is clear, that $\Lambda_g$ includes the following terms
\begin{gather}
 \boldsymbol{u}^T \Lambda_g \boldsymbol{u}= N_g\left(\frac{y''}{y}\mathrm{d}x^2-4 \boldsymbol{r}^T\boldsymbol{u} \right)
 +\frac14
\sum_{\forall\, \mathcal{I}_0\cup \mathcal{J}_0} \left( \mathrm{d}\ln
\frac{\prod\limits_{ i\in \mathcal{I}_0} x-e_i}{\prod\limits_{j\in \mathcal{J}_0} x-e_j} \right)^2.
\label{rational}
\end{gather}
Using the elementary identities
\begin{gather*}
 \sum_{\forall\, \mathcal{I}_0\cup \mathcal{J}_0}\left( \sum_{i\in \mathcal{I}_0}\frac{1}{x-e_i} \right)=
N_g\frac{f'(x)}{2 f(x)},\\
 \sum_{\forall\, \mathcal{I}_0\cup \mathcal{J}_0}\left( \sum_{i\in \mathcal{I}_0}\frac{1}{x-e_i} \right)^2=
N_g\frac{(2g+1)f'(x)^2-(g+1)f(x)f''(x)}{(4g+2) f(x)^2},
\end{gather*}
and the equality
\begin{gather*}
\frac{\mathrm{d}}{ \mathrm{d} x}\ln \frac{\prod\limits_{ i\in \mathcal{I}_0} x-e_i}{\prod\limits_{j\in \mathcal{J}_0} x-e_j} = \sum_{i\in \mathcal{I}_0}\frac{2}{x-e_i}-\frac{f'(x)}{f(x)},
\end{gather*}
we f\/ind after summation over all partitions $\mathcal{I}_0\cup \mathcal{J}_0$,
\begin{gather*}
\sum_{\forall\, \mathcal{I}_0\cup \mathcal{J}_0} \!\!\left( \mathrm{d} \ln
\frac{\prod\limits_{ i\in \mathcal{I}_0} \!x-e_i}{\prod\limits_{j\in \mathcal{J}_0}\! x-e_j} \right)^2\! = 4 N_g\frac{(2g+1)f'(x)^2-(g+1)f(x)f''(x)}{(4g+2) f(x)^2}\mathrm{d}x^2 - N_g\left( \frac{f'(x)}{f(x)}\right)^2\!\mathrm{d}x^2.
\end{gather*}
Furthermore, it is clear that
\begin{gather*}
 \frac{ y''}{ y } = \frac{2 f(x) f''(x) - f'(x)^2 }{4f(x)^2}.
\end{gather*}
Plugging these parts into the right-hand side of \eqref{rational} we get the expression
\begin{gather*}
\frac{N_g}{f(x)^2} \left\{ \frac{(2g+1) f'(x)^2 - (g+1)f(x)f''(x) }{ (4g+2)}\mathrm{d}x^2
+ \frac12 (f(x)f''(x) - f'(x)^2)\mathrm{d}x^2 \right\} - 4N_g \boldsymbol{r}^T\boldsymbol{u}.
\end{gather*}
Because all of the $f'(x)$ and most of the $f(x)$ cancel, this expression shrinks to
\begin{gather}
 \boldsymbol{u}^T \Lambda_g \boldsymbol{u}=\frac{N_g}{(4g+2) f(x)}gf''(x)\mathrm{d}x^2-4N_g\boldsymbol{r}^T\boldsymbol{u}.
\label{cancellation}
\end{gather}
Next, we investigate the term \begin{gather*}\boldsymbol{r}^T\boldsymbol{u}=\frac1{4}\sum_{i=1}^g\left(\sum_{k=i}^{2g+1-i} (k+1-i )\lambda_{k+1+i}x^{k+i-1}\right)\frac{\mathrm{d}x^2}{f(x)}.\end{gather*} Further inspection of this double-sum leads to
\begin{gather}
\boldsymbol{r}^T\boldsymbol{u}=\frac{1}{4}\sum_{k=2}^{2g+2}
\left(\frac1{4} (k-1 )^2+\frac1{8}\big( (-1 )^{k-1}-1\big)\right)
\lambda_k x^{k-2} \frac{\mathrm{d}x^2}{f(x)}.\label{rtu}
\end{gather}
Plugging equation~\eqref{rtu} into equation~\eqref{cancellation} and multiplying with $f(x)=y^2$ gives equation~\eqref{refrelation}, but summed up to $2g+2$. The additional terms for $k=2g+1$ and $k=2g+2$ are zero, hence the statement of Lemma~\ref{reflemma} follows.

With the help of equation~\eqref{refrelation} we can now read of\/f the sum of each anti-diagonal of $\Lambda_g$ by comparing the order of $x$, and we f\/ind that these anti-diagonal sums are integer multiples of a~single~$\lambda_j$.
\end{proof}

\begin{Corollary}
For given genus $g$ let $\varkappa$ be of the form~\eqref{kappaAnsatz2}.
Then the $(k-1)$th anti-diagonal sum of $\Lambda_g$ is given as $\Sigma_{g;k}\lambda_k$ with the integer number
\begin{gather*}\Sigma_{g;k}= \frac{N_g}{4g+2} \left[ \frac{1}{2}k(2g+2-k)+\frac{1}{4}(2g+1)\big((-1)^k-1\big) \right]. \end{gather*}
\end{Corollary}
Partly, the structure of $\Lambda_g$ is obvious: Apparently, the matrix is symmetric along the main diagonal (because~$\varkappa$ is, too). Furthermore the coef\/f\/icients of the $\lambda_j$ are distributed symmetric also along the main anti-diagonal: $\Sigma_{g;k}$ is symmetric along $k=g+1$.

So far, these statements concern the anti-diagonal sums of $\Lambda_g$. As soon as it comes to the single entries we need another method to arrive at reliable claims. Such a method will be derived in Section~\ref{section6} and exemplif\/ied there in a number of examples. Statements for arbitrary genus are under inspection right now, but still we can extrapolate the structure of $\Lambda_g$: The above assumption of linearity with respect to a single $\lambda_j$ leads to the occurrence of the $\lambda_j$ (without partition numbers) in a Hankel-type structure. The distribution of the partition numbers on the other hand is not trivial. We believe that they are given in terms of combinatorial expressions, but by now their pattern is only partly revealed.

Assuming this structure, we can solve for the partition numbers for a given hyperelliptic curve (and thus for all of f\/ixed genus).
The next cases are
\begin{gather*}
\Lambda_4=\left(\begin{matrix}56\lambda_2&21\lambda_3&6\lambda_4&\lambda_5\\
21\lambda_3&72\lambda_4&27\lambda_5&6\lambda_6\\
6\lambda_4&27\lambda_5&72\lambda_6&21\lambda_7\\
\lambda_5&6\lambda_6&21\lambda_7&56\lambda_8
\end{matrix}\right),\qquad
\Lambda_5=\left( \begin{matrix}
210\lambda_2&84\lambda_3&28\lambda_4&7\lambda_5&\lambda_6\\
84\lambda_3&280\lambda_4&119\lambda_5&38\lambda_6&7\lambda_7\\
28\lambda_4&119\lambda_5&300\lambda_6&119\lambda_7&28\lambda_8\\
7\lambda_5&38\lambda_6&119\lambda_7&280\lambda_8&84\lambda_9\\
\lambda_6&7\lambda_7&28\lambda_8&84\lambda_9&210\lambda_{10}
\end{matrix}\right).
\end{gather*}

Further inspection of the matrices lead to the observation: The $k$th entry of the f\/irst row (column) is $\binom{2g+1-k}{g-k}\lambda_{1+k}$ and this row's coef\/f\/icients sum to $\binom{2g+1}{g-1}$. We remark that for odd curves it was shown in \cite{ehkkl11} on examples $g=2,3,4$ that this fact follows from the solution of the Jacobi inversion problem in terms of Klein--Weierstrass $\wp_{i,j}$-functions.

\section[Formulae for entries of $\Lambda_g$]{Formulae for entries of $\boldsymbol{\Lambda_g}$}\label{section6}
So far the full description of the matrix $\Lambda_g$ for arbitrary genus is still open. The main problem within the method presented here is that in equation~\eqref{kw} the vectors $\boldsymbol u$, which appear in the quadratic form $\boldsymbol{ u}^T\varkappa \boldsymbol{u}$, are evaluated at the same point and hence don't lead to a full system of equations. In \cite[Lemma~4.2]{ehkkl11} linear equations were derived for~$\Lambda_{g;i,j}^{[\varepsilon]}$ based on Baker's construction of Abelian functions in terms of Klein--Weierstrass functions~$\wp_{i,j} $ (documented in~\cite{bak897,bak907} and more recently by Buchstaber, Enolski and Leykin in~\cite{bel997} and others). In the derivation the constants~$\Lambda_{g;i,j}^{[\varepsilon]}$ appear as values of the $\wp_{i,j}$-functions for even non-singular half-periods relevant for a~partition~$\mathcal{I}_0$.

A general formula for the integer entries of the matrix $\Lambda_g$ was not found there, but it was understood that f\/inding of each such integer in higher genera using the above mentioned equations is an extremely time consuming procedure. Therefore we suggest to combine the derived formula for the anti-diagonal sums and the aforementioned formulae, which permit us to compute only a part of the entries to $\Lambda_g$ and substantially speed up the whole calculations because of that.

Here we present the derivation of a solvable system of linear equation for $\Lambda_{g;i,j}^{[\varepsilon]}$ which is independent of the multi-variable $\sigma$-functions of Baker's theory. Namely we will prove the following

\begin{Proposition}\label{Lambdaequations} Let a hyperelliptic curve of genus $g$ be realized in the form
\begin{gather*} y^2 = \prod_{m=1}^{2g+2} (x-e_m), \qquad e_m\in\mathbb{C},
\end{gather*} and let $\mathcal{I}_0=\{i_1,\ldots,i_{g+1}\}$ be a partition of branching points with $[\varepsilon]$ the associated characteristic. Let $\boldsymbol{X}_i = (1, e_i,\ldots, e_i^{g-1})^T$, $i\in \mathcal{I}_0$, be $g$-vectors.
Then the following formulae are valid
\begin{gather}
\boldsymbol{X}_i^T\Lambda_g^{[\varepsilon]}\boldsymbol{X}_j =-\frac{F(e_i,e_j)}{(e_i-e_j)^2}
 \equiv \sum_{n=0}^{g} e_i^ne_j^n S_{2g-2n}^{(i,j)},
\qquad i\neq j \in \mathcal{I}_0. \label{lambdaeq}
\end{gather}
Here, $F(x,z)$ is the Kleinian $2$-polar, whilst $S_{k}^{(i,j)}$ are the elementary symmetric functions of order~$k$ built in branching points with indices from the set $\{ 1,2,\ldots, 2g+2 \} / \{ i,j \}$.
 \end{Proposition}
Remark that \eqref{lambdaeq} is analogous to the case of odd curves stated in \cite{ehkkl11} and is backed up with many computer experiments. We omit here the proof of the second identity and place emphasis on the proof of the f\/irst identity, which will be given below.

Writing equations \eqref{lambdaeq} for all possible pairs $\{i,j\} \in \{ i_1,\ldots, i_{g+1} \} $ we get $g(g+1)/2$
equations with respect to $g(g+1)/2$ quantities $\Lambda_{g;i,j}^{[\varepsilon]}$, $1\leq i < j \leq g$. It is convenient to introduce $g(g+1)/2$-vectors,
\begin{align*}
\boldsymbol{\Lambda}_g^{[\varepsilon]} =\left( \begin{matrix} \Lambda_{g;1,1}^{[\varepsilon]} \vspace{1mm}\\ \Lambda_{g;1,2}^{[\varepsilon]}
 \\ \vdots \\ \Lambda_{g;g,g}^{[\varepsilon]}
 \end{matrix} \right),\qquad \boldsymbol{Q}_g^{[\varepsilon]}= \sum_{n=0}^{g} \left(\begin{matrix}
e_{i_1}^{n}e_{i_2}^n S^{(i_1,i_2)}_{2g-2n}\vspace{1mm}\\ e_{i_1}^{n}e_{i_3}^nS^{(i_1,i_3)}_{2g-2n}\\ \vdots \\ e_{i_{g-1}}^{n}e_{i_g}^n
S^{(i_{g-1},i_g)}_{2g-2n}
\end{matrix}\right).
\end{align*}
Due to the symmetry of $\Lambda_g^{[\varepsilon]}$, $\boldsymbol{\Lambda}_g^{[\varepsilon]}$ only includes the entries with ordered indices.

Now the above described equations can be written in the short form
\begin{gather}
M_g^{[\varepsilon]} \boldsymbol{\Lambda}_g^{[\varepsilon]} =\boldsymbol{Q}_g^{[\varepsilon]}, \label{system1}
\end{gather}
where $M_g^{[\varepsilon]}$ is a $ g(g+1)/2\times g(g+1)/2 $-matrix of the following form:
Each row's entry $M_{g; \boldsymbol{\cdot}, k}^{[\varepsilon]}$ belongs to some index $(i,j)$ by the specif\/ic position~$k$ of $\boldsymbol{\Lambda}_{g;i,j}^{[\varepsilon]}$. Now take the matrix
\begin{gather*}\mathcal{M}
=\big(\big[(\boldsymbol{X}_r\cdot\boldsymbol{X}_s^T
+\boldsymbol{X}_s\cdot\boldsymbol{X}_r^T)_{\{r,s\}\in \mathcal{I}_0 }\big]_{i,j}
 \cdot [1-\delta_{i,j}/2]\big)_{i,j=1,\ldots, g(g+1)/2}.\end{gather*} Then we have $M_{g;\;\boldsymbol{\cdot},\; k}^{[\varepsilon]}=\mathcal{M}_{i,j}$. In other words: $\mathcal{M}$ is reshaped to a vector in the
same way as $\Lambda_g^{[\varepsilon]}$ and all these vectors with dif\/ferent $r$ and $s$ constitute to the matrix $M_g^{[\varepsilon]}$.
\begin{Proposition} The linear system~\eqref{system1} with respect to $g(g+1)/2$ entries of the matrix $ \Lambda_g^{[\varepsilon]}$ is solvable.
\end{Proposition}
\begin{proof} Direct calculation shows
\begin{gather*}
\operatorname{Det} M_g^{[\varepsilon]} =\operatorname{Det} \left( \operatorname{Vandermonde} ( e_{i_1},\ldots, e_{i_{g+1}} )\right)^{g-1}\neq 0.
\end{gather*}
Therefore
\begin{gather}
 \boldsymbol{\Lambda}_g^{[\varepsilon]} =\big(M_g^{[\varepsilon]}\big)^{-1}\boldsymbol{Q}_g^{[\varepsilon]}, \label{system2}
\end{gather}
which proves the proposition.
\end{proof}

When $ \boldsymbol{\Lambda}_g^{[\varepsilon]}$ is found according to~\eqref{system2} it remains to sum over all partitions by $g+1$ elements of
$\{ 1, \ldots, 2g+2 \}$ to f\/ind the partition numbers of the anti-diagonal sums we were looking for.

\begin{proof} Now to the proof of Proposition~\ref{Lambdaequations}, which accords with the anonymous referee's suggestion to implement Corollary~2.12 of \cite[p.~28]{fay973}. This corollary represents a remarkable relation between the canonical bi-dif\/ferential $\Omega (P,Q )$ (see equation~\eqref{Bidifferential}) and the Szeg\"o kernel $R\left(P,Q|\boldsymbol{e}\right)$:
\begin{gather} R(P,Q|\boldsymbol{e}) R(P,Q|-\boldsymbol{e}) =\Omega(P,Q) +
 \sum_{i,j=1}^g
 \frac{\partial^2}{\partial z_i\partial z_j} \ln \theta[\boldsymbol{e}] (0) v_i(P)v_j(Q).\label{Cor2.12}
 \end{gather}
Here, $P=(x,y)$, $Q=(z,w)$ are points on the curve, $\boldsymbol{e}$ is a vector for which $\theta(\boldsymbol{e})\neq 0$, and with $E(P,Q)$ being the Schottky--Klein prime-form, the Szeg\"o kernel is given as
\begin{gather*}
R(P,Q|\boldsymbol{e})=\frac{ \theta[\boldsymbol{e}] \left( \int_{P}^Q \boldsymbol{v} \right) }{ \theta[\boldsymbol{e}](0) E(P,Q) }.
\end{gather*}
Our purpose demands $\boldsymbol{e}$ to be a non-singular \emph{even} half period associated to the partition $\mathcal{I}_0=\{ i_1,\ldots, i_{g+1} \}$; we will denote further $[\boldsymbol{e}]=[\varepsilon]$.
Fay states \cite[p.~13]{fay973} an algebraic representation of the Szeg\"o kernel of a hyperelliptic curve associated to non-singular even half-periods:
\begin{gather*}
R(P,Q|\varepsilon)=\frac{ w \prod\limits_{k=1}^{g+1} (x-e_{i_k}) + y \prod\limits_{k=1}^{g+1} (z-e_{i_k}) }{2(x-z)}
\left[ \frac{\mathrm{d}x\mathrm{d} z}{yw \prod\limits_{k=1}^{g+1} (x-e_{i_k})(z-e_{i_k})} \right]^{1/2}.
\end{gather*}
Note that in \eqref{Cor2.12} all terms are evaluated at two points $P$ and $Q$ and therefore presumably deliver a full system. Moreover, for even half periods the second logarithmic derivative of $\theta$ reduces to~$\theta''/\theta$. Together, \eqref{Cor2.12} takes the form{\samepage
\begin{gather}
 \frac{\left( w \prod\limits_{k=1}^{g+1} (x-e_{i_k}) + y \prod\limits_{k=1}^{g+1} (z-e_{i_k})\right)^2 }{4(x-z)^2} \frac{1}{ \prod\limits_{k=1}^{g+1} (x-e_{i_k})(z-e_{i_k})}\frac{\mathrm{d} x \mathrm{d} z}{ yw } \nonumber\\
\qquad{}=\frac{F(x,z)+2yw}{4(x-z)^2}\frac{\mathrm{d} x \mathrm{d} z}{ yw }+2\sum_{i,j=1}^g \varkappa_{ij}u_i(P)u_{j}(Q)+\sum_{i,j=1}^g \frac{\theta_{i,j}[\varepsilon]}{\theta[\varepsilon]} v_i(P)v_j(Q),
 \label{relation}
\end{gather}
where the algebraic representation of the canonical bi-dif\/ferential was used.}

As it was worked out in Sections~\ref{section4} and~\ref{section5}, $\varkappa$ has the following form
\begin{gather*}
\varkappa=-\frac{1}{2} \frac{1} {\Theta[\varepsilon]}
\left( \begin{matrix} \Theta_{1,1}[\varepsilon]& \Theta_{1,2}[\varepsilon]&\ldots
 &\Theta_{1,g}[\varepsilon] \\
\vdots&\vdots&\ldots&\vdots\\
\Theta_{1,g}[\varepsilon]& \Theta_{2,g}[\varepsilon]&\ldots &\Theta_{g,g}[\varepsilon]
 \end{matrix} \right)+ \frac18 \Lambda_g^{[\varepsilon]}.
\end{gather*}
This is the ``unsummed'', characteristic dependent representation \eqref{kappaAnsatz} of~$\varkappa$, hence the mat\-rix~$\Lambda_g$ has the index~$[\varepsilon]$, indicating it consists of symmetric functions like it was in equations~\eqref{example2} and~\eqref{example3}.

Because of the representation of $\varkappa$ in this form, the last two terms on the right-hand side of equation~\eqref{relation} are reduced to
\begin{gather*} 2\sum_{i,j=1}^g \varkappa_{i,j}u_i(P)u_{j}(Q)
+\frac{1}{ \theta[\varepsilon ]} \sum_{i,j=1}^g \frac{\partial^2}{\partial z_i\partial z_j}
 \left. \theta[\varepsilon](\boldsymbol{z})\right\vert_{\boldsymbol{z}=0} v_i(P)v_j(Q) =
\frac1{4}\sum_{i,j=1}^g \Lambda_{g;ij}^{[\varepsilon]}u_i(P)u_{j}(Q).\end{gather*}

Using $y(e_r)=w(e_s)=0$ and the divisibility of $F(e_r,e_s)$ by $(e_r-e_s)^2$ stated above we can write~\eqref{relation} in the form
\begin{gather}
 \frac{\left( w \prod\limits_{k=1}^{g+1} (x-e_{i_k}) + y \prod\limits_{k=1}^{g+1} (z-e_{i_k})\right)^2 }{4(x-z)^2} \frac{1}{ \prod\limits_{k=1}^{g+1} (x-e_{i_k})(z-e_{i_k})}\frac{\mathrm{d} x \mathrm{d} z}{ yw } \nonumber\\
\qquad {} =-\frac14\sum_{n=0}^{g} e_r^ne_s^n S_{2g-2n-1}^{(r,s)}\frac{\mathrm{d}x\mathrm{d}z}{yw}+\frac1{4}\sum_{i,j=1}^g \Lambda_{g;r,s}^{[\varepsilon]}u_i(P)u_{j}(Q).\label{relation1}
\end{gather}

Expanding $x=e_r+\xi^2$, $z=e_s+\xi^2$, $\xi\sim 0$, in the vicinity of the branching points, the left-hand side of \eqref{relation1} vanishes up to order~$\xi^2$ whilst the right-hand side leads to the equation
claimed in Proposition~\ref{lambdaeq}
\begin{gather*}
\sum_{i,j=1}^g \Lambda_{g;i,j}^{[\varepsilon]} e_r^{i-1}e_s^{j-1} = \sum_{n=0}^{g} e_r^ne_s^n S_{2g-2n}^{(r,s)}.\tag*{\qed}
\end{gather*}
\renewcommand{\qed}{}
\end{proof}

\begin{Example} Let $g=3$. In this case $\Lambda_3$ can easily be found by using the formula for anti\-dia\-gonal sums and our knowledge of the f\/irst row's structure. Hence we can check the correctness of the equations~\eqref{system1}. To do that we f\/irst f\/ix the partition
$\mathcal{I}_0=\{1,2,3,4\}$ and denote as $[\varepsilon]$ the corresponding characteristic.
Equations~\eqref{system1} are
\begin{gather}
 \left(\begin{matrix} 1&e_1+e_2&e_1^2+e_2^2&e_1e_2&e_1e_2(e_1+e_2)&e_1^2e_2^2\\
 1&e_1+e_3&e_1^2+e_3^2&e_1e_3&e_1e_3(e_1+e_3)&e_1^2e_3^2\\
 1&e_1+e_4&e_1^2+e_4^2&e_1e_4&e_1e_4(e_1+e_4)&e_1^2e_4^2\\
 1&e_2+e_3&e_2^2+e_3^2&e_2e_3&e_2e_3(e_2+e_3)&e_2^2e_3^2\\ 1&e_2+e_4&e_2^2+e_4^2&e_2e_4&e_2e_4(e_2+e_4)&e_2^2e_4^2\\ 1&e_3+e_4&e_3^2+e_4^2&e_3e_4&e_3e_4(e_3+e_4)&e_3^2e_4^2
 \end{matrix} \right) \left( \begin{matrix}
 \Lambda_{3;1,1}^{[\varepsilon]}\\ \Lambda_{3;1,2}^{[\varepsilon]}\\ \Lambda_{3;1,3}^{[\varepsilon]}\\ \Lambda_{3;2,2}^{[\varepsilon]}\\ \Lambda_{3;2,3}^{[\varepsilon]}\\ \Lambda_{3;3,3}^{[\varepsilon]}
 \end{matrix} \right)= \left( \begin{matrix}
 Q_{1,2}^{[\varepsilon]}\\ Q_{1,3}^{[\varepsilon]}\\ Q_{1,4}^{[\varepsilon]}\\ Q_{2,3}^{[\varepsilon]}\\ Q_{2,4}^{[\varepsilon]}\\ Q_{3,4}^{[\varepsilon]}
 \end{matrix} \right),\label{system1g3}
\end{gather}
where
\begin{gather*}
Q_{i_k,i_l}=e_{i_k}^3e_{i_l}^3 + S_2^{(i_k,i_l)}e_{i_k}^2e_{i_l}^2 + S_4^{(i_k,i_l)}e_{i_k}e_{i_l}+ S_6^{(i_k,i_l)}.
\end{gather*}
$\operatorname{Det} M_{3}^{[\varepsilon]} = \operatorname{Vandermonde}
(e_1,e_2,e_3,e_4)^2\neq 0$ and therefore the system \eqref{system1g3} is solvable. In particular,
\begin{gather*} \Lambda_{3;3,3}^{[\varepsilon]}=e_1e_2+e_3e_4+e_5e_6+e_7e_8+(e_1+e_2)(e_3+e_4)+(e_5+e_6)(e_7+e_8). \end{gather*}
The other entries of $\Lambda_{3}^{[\varepsilon]}$ can be derived in the same way. Summing over all $( {}^{8}_4 )=70$ partitions we see that $\Lambda_3$ indeed has the structure as it was claimed in~\eqref{kappa3}.
\end{Example}

\begin{Example} Consider now $\Lambda_5$. Using the anti-diagonal sums, the symmetry and the known structure of the f\/irst row
and column we f\/ind
\begin{gather*}
 \Lambda_5=\left(
\begin{matrix}
210\lambda_2&84\lambda_3&28\lambda_4&7\lambda_5&\lambda_6\\
84\lambda_3&280\lambda_4&119\lambda_5&X\lambda_6&7\lambda_7\\
28\lambda_4&119\lambda_5&Y\lambda_6&119\lambda_7&28\lambda_8\\
7\lambda_5&X\lambda_6&119\lambda_7&280\lambda_8&84\lambda_9\\
\lambda_6&7\lambda_7&28\lambda_8&84\lambda_9&210\lambda_{10}
\end{matrix}
\right).
\end{gather*}
To complete the calculations it is enough to f\/ind the integers $X$ and~$Y$. To calculate $X$ f\/ix a~partition, e.g.,
$ \mathcal{I}_0= \{1,2,3,4,5,6 \}$, and denote with $s_j$ the elementary symmetric functions of the branching
points with indices from the set $\mathcal{I}_0$ and denote $\widetilde{s}_j$ the same functions from
the complementary set. Solving equations~\eqref{system1} for this partition we f\/ind
\begin{gather*} X = 2s_6+2\widetilde{s}_6+s_1\widetilde{s}_5+s_5\widetilde{s}_1. \end{gather*}
Summation over all 924 partitions leads to the previously found value~$38\lambda_6$. To f\/ind $Y$ we can use the
anti-diagonal sum: $2+2X+Y=378$.

This example demonstrates the advantage of combining the methods of Sections~\ref{section5} and~\ref{section6}: instead of summing over~924 partitions of all 15 solutions of system~\eqref{system1} it is enough to do that only with one from them, $\Lambda_{2,4}$.
\end{Example}

\begin{Example} Proceeding as described in this section we found
\begin{gather*} \Lambda_6= \left( \begin{matrix}
729\lambda_2&330\lambda_3&120\lambda_4&36\lambda_5&8\lambda_6&\lambda_7\\
330\lambda_3&1080\lambda_4&492\lambda_5&184\lambda_6&51\lambda_7&8\lambda_8\\
120\lambda_4&492\lambda_5&1200\lambda_6&542\lambda_7&184\lambda_8&36\lambda_9\\
36\lambda_5&184\lambda_6&542\lambda_7&1200\lambda_8&492\lambda_9&120\lambda_{10}\\
8\lambda_6&51\lambda_7&184\lambda_8&492\lambda_9&1080\lambda_{10}&330\lambda_{11}\\
\lambda_7&8\lambda_8&36\lambda_9&120\lambda_{10}&330\lambda_{11}&792\lambda_{12}
 \end{matrix} \right). \end{gather*}
\end{Example}

\section{Choice of basis}\label{section7}

\looseness=-1
In this paper we used a basis due to Baker for a $2g$-dimensional space of singular cohomologies. Our results were derived and numerically tested within this basis. But there is a certain freedom of choice for this basis, namely we are able to add a linear combination of holomorphic dif\/ferentials to the meromorphic dif\/ferentials without changing the Legendre relation. As a~result $\varkappa$ changes by a certain matrix (see \cite[p.~328]{bak897}): If $\boldsymbol{r}\rightarrow \boldsymbol{r}-C\cdot \boldsymbol{u}$ with a symmetric $g\times g$ mat\-rix~$C$, then $\varkappa\rightarrow\varkappa +C/2$. Consequently, if we f\/ix $C=-\frac1{4}\frac{1}{N_g}\Lambda_g$, our main result will change to
\begin{gather}
\varkappa= -\frac1{2}\frac1{N_g}\sum_{N_g [\varepsilon]}\frac{(\Theta_{ab}[\varepsilon])_{a,b=1,\ldots,g}}{\Theta[\varepsilon]}.
\label{Kleinian}
\end{gather}
Of course, for this new basis of meromorphic dif\/ferentials, one has to know explicitly the mat\-rix~$\Lambda_g$. The change of basis is in that sense only a~reformulation of Theorem~\ref{maintheorem}. But it is important to us to point out the existence of such a basis.

Formula \eqref{Kleinian} f\/irst appears in F.~Klein~\cite{klein886,klein888}, then it was recently revisited in a more general context by Korotkin and Shramchenko~\cite{ksh12}. Let us call a cohomology basis $\mathcal{B}''= \{u_1'',\ldots,u_g'',r_1'',\ldots,r_g'' \}$
 a {\it Klein basis}, if the associated $\varkappa$-matrix has a representation of the form~\eqref{Kleinian}.

Our development results in the following
\begin{Proposition}
$\varkappa$ is of the form \eqref{Kleinian} in the Klein cohomology basis
\begin{gather*}
u_i''=u_i,\qquad r_i''=r_i-\sum_j C_{ij}u_j,
\end{gather*}
with $C=-\frac{1}{4}\frac{1}{N_g}\Lambda_g$ and $\Lambda_g$ as in Theorem~{\rm \ref{maintheorem}}.
\end{Proposition}

In this note we presented an explicit construction of a special $2g$-dimensional cohomology basis with periods $(2\omega,2\omega',2\eta,2\eta')$ satisfying the generalized Legendre relation. In this basis the symmetric matrix $\varkappa=\eta(2\omega)^{-1}$ takes a simple form. $\varkappa$ is important for def\/ining the multi-variable $\sigma$-function, and therefore for resolving of the Jacobi--Riemann inversion problem and description of the Jacobi variety of associated curves.

Remarkably, in the special case of $g=2$ Baker presented the matrix $\Lambda_2$ (see \cite[p.~90]{bak907}) in his construction of $\sigma$-expansions. Our result can be seen as a generalization of these results to higher genera hyperelliptic curves. But as he notes there, taking $\varkappa$ in its original form (in the \emph{Baker basis}) leads to simplif\/ications of associated dif\/ferential equations of the $\wp$-functions. Taking this into account, one can choose one's basis by means which is more convenient in the given purpose. We plan to elucidate the nature of these simplif\/ications in more detail in a~subsequent work.

\subsection*{Acknowledgements}
The author is grateful to V.~Enolski for useful discussion and constant interest to the work, and also to all referees, whose comments promoted a further improvement of the text. In especially the author wants to mention the contribution of the anonymous referee, who reported formula~\eqref{refrelation} and reminded us of Fay's Corollary 2.12~\cite{fay973}, which essentially improved our initial statements.
Also the author gratefully acknowledges the Deutsche Forschungsgemeinschaft (DFG) for f\/inancial support within the framework of the DFG Research Training group 1620 Models of gravity.

\pdfbookmark[1]{References}{ref}
\LastPageEnding

\end{document}